\documentclass[11pt, a4paper, twoside]{article}
\usepackage[toc,page]{appendix}
\usepackage[utf8]{inputenc}
\usepackage[british, UKenglish]{babel}

\usepackage{mathtools}
\usepackage{stackrel}
\usepackage{amsmath}
\usepackage{mathrsfs}
\usepackage{hyperref}
\usepackage{cleveref}
\usepackage{nameref}
\usepackage{amssymb}
\usepackage{amsthm}
\usepackage{tikz,float}
\usetikzlibrary{automata,positioning,decorations.pathmorphing}
\usepackage{caption}
\usepackage{tikz-cd}
\usetikzlibrary{arrows}
\tikzset{
commutative diagrams/.cd,
arrow style=tikz,
diagrams={>=latex}}
\usetikzlibrary{decorations.markings}
\tikzset{
commutative diagrams/.cd,
arrow style=tikz,
diagrams={>=latex}}
\usepackage{lmodern}    
\usepackage{ifthen} 
\usepackage{enumitem}
\usepackage{csquotes}
\usepackage{url}
\usetikzlibrary{calc}
\usepackage{verbatim}

\usepackage[backend=bibtex8,
style=numeric,
bibencoding=ascii
]{biblatex}
\newcommand{\Mpres}[2]{\operatorname{Mon}\bigl\langle #1\:|\:#2 \bigr\rangle}

\newcommand{\Gpres}[2]{\operatorname{Gp}\bigl\langle #1\:|\:#2 \bigr\rangle}

\newcommand{\rr}[1]{\operatorname{}\langle #1 \rangle}

\newcommand{\N}{\mathbb{N}}
\newcommand{\Z}{\mathbb{Z}}

\newcommand{\G}{G_{\Gamma}}

\newcommand{\ga}{\Gamma}

\theoremstyle{plain}

\theoremstyle{definition}
\newtheorem*{theorem*}{Theorem}

\newtheorem{theorem}{Theorem}[section]
\newtheorem{prop}[theorem]{Proposition}
\newtheorem{lemma}[theorem]{Lemma}
\newtheorem{cor}[theorem]{Corollary}
\newtheorem{remark}[theorem]{Remark}
\newtheorem{notation}[theorem]{Notation}

\theoremstyle{definition}
\newtheorem{mydef}[theorem]{Definition}
\newtheorem{example}[theorem]{Example}

\author{}

\DeclareMathOperator{\Link}{Lk}
\DeclareMathOperator{\pref}{pref}
\DeclareMathOperator{\supp
}{Supp}

\newcommand{\verteq}{\rotatebox{90}{$\,=$}}

\title{Twisted right-angled Artin groups}
\author{Islam Foniqi}
\date{}
\usepackage{fancyhdr}
\makeatletter
\let\runauthor\@author
\let\runtitle\@title
\makeatother
\begin{document}
\maketitle
\begin{abstract}

We present a complete rewriting system for twisted right-angled Artin groups. Utilizing the normal form coming from the rewriting system, we provide applications that illustrate differences and similarities with right-angled Artin groups, both at the geometric and algebraic levels.
\end{abstract}
	
\renewcommand{\thefootnote}{\fnsymbol{footnote}} 
\footnotetext{
\noindent\emph{MSC 2020 classification:} 05A15, 05C25, 20F36. \newline
\emph{Key words:} Twisted right-angled Artin groups, growth functions, growth series, Cayley graphs}

\section{Introduction}

The subject of this article is the class of twisted right-angled Artin groups (tRAAGs shortly),
whose presentation is given by generators and relations: there are finitely many generators and there is at most one relation between any two distinct generators $a,b$; this relation can be a commutation~$ab = ba$, or a so-called {\it Klein relation} $aba = b$. \\
As basic and important examples of tRAAGs, we have both the fundamental groups of a torus:~$\Z^2 = \Gpres{a, b}{ab = ba}$, and of the Klein bottle:~$K = \Gpres{a, b}{aba = b}$.

The class of tRAAGs generalizes the class of right-angled Artin groups (RAAGs) which plays an important role in geometric group theory (see \cite{charney2007introduction} for a survey).

The class of tRAAGs appears in \cite{clancy2010homology}, where the notion of a \emph{twisted Artin group} was
introduced. This current article provides a generalization of Chapter 3 of the author's PhD thesis~\cite{foniqi2022}. Some instances of tRAAGs appear also in \cite{phdthesis} and \cite{blumer2023oriented}, where the name \emph{oriented} is used instead of twisted.
Recently in \cite{cassella2024hypercubical} the geometry of tRAAGs was studied, and it was shown that tRAAGs are hypercubical groups.

One way to give presentations of groups appearing here, is using graphs; the vertices of the graph represent the generators of the group, while edges encode the relations. 

\begin{mydef}\label{def_simplicial_graph}
A {\it simplicial graph $\Gamma$} is a pair $\Gamma = (V, E)$, where $V = V\ga$ is a set whose elements are called vertices, and~$E = E\ga \subseteq \{\{x,y\} \mid x,y \in V, x\neq y\}$ is a set of paired distinct vertices, whose elements are called edges.
\end{mydef}
Note that in simplicial graphs there are no multiple edges and no loops. Usually, the
presentation of RAAGs is encoded in such graphs, where an edge connecting generators~$a$ and~$b$ gives rise to the commutation~$ab = ba$.

To define tRAAGs we use mixed graphs, which are similar to simplicial graphs but directed edges are allowed.
We define mixed graphs using simplicial graphs as an underlying structure.



\begin{mydef}\label{def_mixed_graph}
A {\it mixed graph} $\Gamma$ consists of an {\it underlying simplicial graph}~$(V, E)$, a set of {\it directed edges} $D \subseteq E$, and two maps $o, t \colon D \longrightarrow V$. We denote it as $\Gamma = (V, E, D, o, t)$. \\
For an edge $e=\{x,y\} \in D$, the maps $o, t$ satisfy $o(e), t(e) \in \{x,y\}$, and $o(e) \neq t(e)$. Refer to~$o(e)$ and $t(e)$ as the {\it origin} and the {\it terminus} of edge $e$ respectively.
\end{mydef}

\begin{notation} 
Let $\Gamma = (V, E, D, o, t)$ be a mixed graph.
\begin{itemize}[itemsep=4pt,parsep=0pt,topsep=0pt, partopsep=0pt]
    \item[(i)] If $e = \{a,b\}$ is an undirected edge, i.e. $e \in E \setminus D$, we write $e = [a, b]$; note that also~$e = [b, a]$.
    \item[(ii)] Instead, if $e = \{a,b\}$ is a directed edge, i.e. $e \in D$, we write $e = [o(e), t(e)\rangle$. In this case, either $e = [a, b\rangle$, or $e = [b, a\rangle$.
\end{itemize}
Geometrically, we present the respective edges $[a,b]$, and $[a,b\rangle$ as follows:
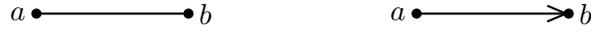
\begin{figure}[H]
\centering
\begin{tikzpicture}[>={Straight Barb[length=7pt,width=6pt]},thick]

\draw[fill=black] (0,0) circle (1.5pt) node[left] {$a$};
\draw[fill=black] (2,0) circle (1.5pt) node[right] {$b$};
\draw[fill=black] (5,0) circle (1.5pt) node[left] {$a$};
\draw[fill=black] (7,0) circle (1.5pt) node[right] {$b$};
\draw[thick] (0,0) -- (2,0);
\draw[thick, ->] (5,0) --  (7,0);
\end{tikzpicture}\caption{Types of edges in $E\Gamma$}\label{types_of_edges}
\end{figure}
Since every edge in $E = E\Gamma$ has exactly one of the types appearing in Figure \ref{types_of_edges}, one can express the set of edges in $\Gamma$ as 
$$E\Gamma = \overline{E\Gamma} \sqcup \overrightarrow{E\Gamma},$$
where $\overline{E\Gamma} = E \setminus D$ consists of undirected edges, and $\overrightarrow{E\Gamma} = D$ consists of directed edges.
\end{notation}
When dealing with group presentations encoded from graphs, we want edge $[a,b]$ to represent the commutation of $a,b$; and edge $[a,b\rangle$ to represent the corresponding Klein relation; hence by a slight abuse of notation, we also set: $[a,b] = aba^{-1}b^{-1}$, and~$[a,b\rangle = abab^{-1}$.

\begin{mydef}\label{def_of_traags}
Let $\Gamma = (V,E)$ be a mixed graph with $E = \overline{E\Gamma} \sqcup \overrightarrow{E\Gamma}$. Define a group~$G_{\Gamma}$, corresponding to $\Gamma$ as
$$G_{\Gamma} = \langle  V \mid ab = ba \text{ if $[a,b]\in \overline{E\Gamma}$ }, aba = b \text{ if $[a,b\rangle \in \overrightarrow{E\Gamma}$ }\rangle.$$
We call $G_{\Gamma}$ the {\it twisted right-angled Artin group} based on $\Gamma$, and we call $\Gamma$ {\it the defining graph} of $G_{\Gamma}$. If $\overrightarrow{E\Gamma} = \emptyset$ (equivalently if $E\Gamma = \overline{E\Gamma}$), $G_{\Gamma}$ is a {\it right-angled Artin group}.
\end{mydef}


In Section \ref{sec: overview_rewriting_systems}, we review the necessary preliminaries on rewriting systems. Then in Section \ref{sec: rewriting_systems_tRAAGs} we provide a complete rewriting system for tRAAGs. In Section \ref{sec: normal_form} we provide a normal form theorem for tRAAGs, and we consider some of its applications in Section \ref{applications_normal_form}.

\section{An overview of rewriting systems}\label{sec: overview_rewriting_systems}
For this section we follow Chapter 12 of \cite{holt2005handbook}, where the reader can find more details.

Let $A$ be a finite set, which we call {\it alphabet}. Denote by $A^{\ast}$ the {\it free monoid} over $A$. Elements of $A^{\ast}$ are finite sequences which we call {\it words}. Usually we denote the empty word by $1$, which serves as the identity of the monoid.

\begin{mydef}
By a {\it rewriting system} on~$A^{\ast}$ we mean a set~$\mathcal{S}$ of ordered pairs~$(l, r)$ where~$l, r$ are words in~$A^{\ast}$.
\end{mydef}

So, $\mathcal{S} \subseteq A^{\ast} \times A^{\ast}$;  elements of $\mathcal{S}$ are called rewrite-rules, $l$ is the left-hand side, and~$r$ is the right-hand side of the rule $(l, r)$.
In words over $A$, one can replace occurrences of $l$ by $r$. Additionally we assume that no two distinct rules have the same left-hand sides.

For $u, v \in A^{\ast}$ we write $u \rightarrow_{_\mathcal{S}} v$, and say that the word $u$ reduces to the word $v$, if there are $p,q \in A^{\ast}$ such that $u = plq$ and $v = prq$ with $(l, r) \in \mathcal{S}$; i.e. $v$ is obtained from $u$ from a single application of a rewrite-rule. If there is no ambiguity, we omit the subscript $\mathcal{S}$.

We say that~$u \in A^{\ast}$ is {\it $\mathcal{S}$-irreducible} (or {\it $\mathcal{S}$-reduced}) if there are no strings~$v \in A^{\ast}$ satisfying~$u \rightarrow v$. Again, we can omit  $\mathcal{S}$, and talk about irreducible (reduced) words.

\begin{example}\label{ex: rewriting_ba_ab}
Let $A = \{a, b\}$, and $\mathcal{S} = \{(ba, ab)\}$. One has the following reductions:
\[
\underline{ba}ba \rightarrow \underline{ab}ba,\quad 
ab\underline{ba} \rightarrow ab\underline{ab},\quad
a\underline{ba}b \rightarrow a\underline{ab}b
\]
Note that $aabb$ is reduced, as we cannot apply any rewrite-rules.
\end{example}

Define $\rightarrow^{\ast}$ to be the reflexive, transitive closure of $\rightarrow$. So, for $u, v \in A^{\ast}$ writing $u \rightarrow^{\ast} v$ means that there is a sequence of words $u_0, \ldots, u_n$ in $A^{\ast}$ with
\[
u=u_0 \rightarrow u_1 \rightarrow \dots \rightarrow u_{n-1} \rightarrow u_n = v.
\]
The rewriting system $\mathcal{S}$ is called {\it Noetherian} (or {\it terminating}) if there is no infinite chain of words~$u_i \, (i \geq 0)$ with $u_i \rightarrow u_{i+1}$ for all $i$. In terminating systems, for any $u \in A^{\ast}$, there is an irreducible word $v \in A^{\ast}$ such that $u \rightarrow^{\ast} v$.
The system in Example \ref{ex: rewriting_ba_ab} is Noetherian as words get smaller alphabetically while preserving the length.

Say that a rewriting system $\mathcal{S}$ is {\it confluent} if any two words with a common ancestor have a common descendent; i.e. whenever $u \rightarrow^{\ast} v_1$, and  $u \rightarrow^{\ast} v_2$ with $u, v_1, v_2 \in A^{\ast}$, there exists~$w \in A^{\ast}$ such that $v_1 \rightarrow^{\ast} w$ and $v_1 \rightarrow^{\ast} w$; expressed diagrammatically as:
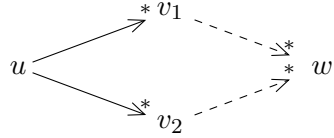
\begin{center}
\begin{tikzpicture}[>={Straight Barb[length=6pt,width=5pt]}, inner sep=2pt]
    \node at (0,0) (u) {$u$}; 
    \node at (2,0.75) (v1) {$\, v_1\,$}; 
    \node at (2,-0.75) (v2) {$\, v_2\,$}; 
    \node at (4,0) (w) {$\quad w \quad$}; 
    
    \path[->] (u) edge node[above, pos=1] {$_*$} (v1);
 	\path[->] (u) edge node[above, pos=1] {$_*$} (v2); 
 	\path[->, dashed] (v1) edge node[above, pos=1] {$_*$} (w);
 	\path[->, dashed] (v2) edge node[above, pos=1] {$_*$} (w);
\end{tikzpicture}
\end{center}
In confluent systems, any word $w \in A^{\ast}$ has at most one irreducible descendent.

A system $\mathcal{S}$ is called {\it complete} if it is both Noetherian and confluent. In complete systems, for any word $w \in A^{\ast}$ there exists exactly one irreducible descendent
$z$ of $w$ which we call the {\it normal form} of $w$. 

A system $\mathcal{S}$ is called {\it locally confluent}  if for all $u \in A^{\ast}$ with $u \rightarrow v_1$ and  $u \rightarrow v_2$ for some~$v_1, v_2 \in A^{\ast}$, there exists~$w \in A^{\ast}$ such that $v_1 \rightarrow^{\ast} w$ and $v_2 \rightarrow^{\ast} w$; expressed by the following diagram:
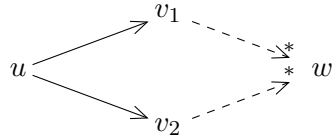
\begin{center}
\begin{tikzpicture}[>={Straight Barb[length=6pt,width=5pt]}, inner sep=2pt]
    \node at (0,0) (u) {$u$}; 
    \node at (2,0.75) (v1) {$v_1\,$}; 
    \node at (2,-0.75) (v2) {$v_2\,$}; 
    \node at (4,0) (w) {$\quad w \quad$}; 
    
    \path[->] (u) edge  (v1);
 	\path[->] (u) edge  (v2);
 	\path[->, dashed] (v1) edge node[above, pos=1] {$_*$} (w);
 	\path[->, dashed] (v2) edge node[above, pos=1] {$_*$} (w);
\end{tikzpicture}
\end{center}
System $\mathcal{S}$ is  complete if and only if it is terminating and locally confluent (Lemma 12.15 \cite{holt2005handbook}).

The symmetric closure of $\rightarrow^{\ast}$ is denoted by $\leftrightarrow^{\ast}$. 
If $\mathcal{S}$ is complete then each equivalence class under $\leftrightarrow^{\ast}$ contains a unique irreducible element (Lemma 12.16 \cite{holt2005handbook}), i.e. equivalent words have common descendants. The system in Example \ref{ex: rewriting_ba_ab} is complete with set of normal forms~$\mathcal{N} = \{a^ib^j|i, j \in \N_0\}$.

Determining whether a rewriting system is terminating is undecidable, hence completeness is difficult to tackle. Knuth and Bendix provided an algorithm which takes as input a finite terminating system~$\mathcal{S}$, tests whether it is complete, and if it is not, then new rules are added to complete it. 
The algorithm tests local confluence, as for Noetherian systems local confluence is equivalent to confluence.

\begin{mydef}
Two rewriting systems $\mathcal{S}$ and $\mathcal{S}'$ on $A^\ast$ are called {\it equivalent} if $u \leftrightarrow^{\ast}_{\mathcal{S}} v$ if and only if $u \leftrightarrow^{\ast}_{\mathcal{S'}} v$ for $u, v \in A^\ast$.
\end{mydef}

\begin{remark}\label{rem: len-lex ordering}
In rewriting systems appearing in this article, any rule $(l, r)$ satisfies the len-lex order, i.e. one of the following holds:
\begin{itemize}[topsep=3pt, partopsep=0pt, itemsep = 0pt]
\item $l$ has shorter length than $r$, or
\item $l$ has the same length as $r$ but $r$ is smaller in the alphabetical order. 
\end{itemize}
For example, if $A = \{a, b\}$ and we decide to order~$A$ by~$b < a$, then the first few strings in the associated shortlex order are
\[
1, b, a, bb, ba, ab, aa, bbb, \cdots
\]
\end{remark}
In particular, our systems are Noetherian, so 
to show their completeness, it is enough to show that they are locally confluent.

\begin{lemma}[Lemma 12.17 \cite{holt2005handbook}]\label{lem: locally confluent}
The rewriting system $\mathcal{S}$ is locally confluent if and only if for all pairs of rules $(l_1, r_1), (l_2, r_2) \in \mathcal{S}$, the following
conditions are satisfied:
\begin{itemize}[topsep=3pt, partopsep=0pt, itemsep = 0pt]
\item[(i)] If $l_1 = us$ and $l_2 = sv$ with $u, s, v \in A^{\ast}$ and $s \neq \epsilon$, then there exists $w \in A^{\ast}$ that satisfies~$r_1 v \rightarrow ^{\ast} w$, and $u r_2 \rightarrow ^{\ast} w$.
\begin{center}
\begin{tikzpicture}[>={Straight Barb[length=6pt,width=5pt]}, inner sep=2pt]
    \node at (0,0) (u) {$usv$}; 
    \node at (2,0.5) (v1) {$\, r_1v\,$}; 
    \node at (2,-0.5) (v2) {$\, ur_2\,$}; 
    \node at (4,0) (w) {$\quad w \quad$}; 
    
    \path[->] (u) edge (v1);
 	\path[->] (u) edge (v2); 
 	\path[->, dashed] (v1) edge node[above, pos=1] {$_*$} (w);
 	\path[->, dashed] (v2) edge node[above, pos=1] {$_*$} (w);
\end{tikzpicture}
\end{center}
\item[(ii)] If $l_1 = usv$ and $l_2 = s$ with $u, s, v \in A^{\ast}$ and $s \neq \epsilon$, then there exists $w \in A^{\ast}$ that satisfies~$r_1 \rightarrow^{\ast} w$, and $ur_2v \rightarrow^{\ast} w$.
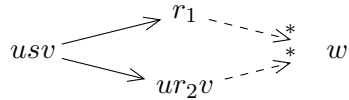
\begin{center}
\begin{tikzpicture}[>={Straight Barb[length=6pt,width=5pt]}, inner sep=2pt]
    \node at (0,0) (u) {$usv$}; 
    \node at (2,0.5) (v1) {$\, r_1\,$}; 
    \node at (2,-0.5) (v2) {$\, ur_2v\,$}; 
    \node at (4,0) (w) {$\quad w \quad$}; 
    
    \path[->] (u) edge (v1);
 	\path[->] (u) edge (v2); 
 	\path[->, dashed] (v1) edge node[above, pos=1] {$_*$} (w);
 	\path[->, dashed] (v2) edge node[above, pos=1] {$_*$} (w);
\end{tikzpicture}
\end{center}
\end{itemize}
\end{lemma}

A pair of rules satisfying either of the two conditions of Lemma \ref{lem: locally confluent} is called a {\it critical pair}; refer to case (i) as an {\it overlap pair} and to (ii) as an {\it inclusion pair}. If one of these conditions fail (so $\mathcal{S}$ is not confluent), we obtain two distinct words $w_1, w_2$ that are irreducible and equivalent under $\leftrightarrow^{\ast}$. One can resolve this instance of incompleteness by adjoining either~$(w_1, w_2)$ or~$(w_2, w_1)$ as a new rule to $\mathcal{S}$. We can then continue to check for local confluence.

\subsection{The Knuth-Bendix completion process}
The procedure of examining critical pairs and adjoining new rules to a rewriting system if necessary, is known as the {\it Knuth-Bendix completion process}. It may happen that this process eventually terminates with a finite complete set, which is what one would like. In general however, the process does not terminate, but it will generate a complete system with an infinite set of rules. In our examples, such an infinite set has a transparent structure, which makes it almost as useful as a finite complete set.

It is not straight forward in general to decide which of $(w_1, w_2)$ or $(w_2, w_1)$ to adjoin to~$\mathcal{S}$ in order to resolve a critical pair. In this article we resolve this problem by imposing a well-ordering~$\leq$ on $A^{\ast}$ that has the property that $u \leq v$ implies $uw \leq vw$ and $wu \leq wv$ for all~$w \in A^{\ast}$. Such an ordering is known as a {\it reduction ordering}. Then we set the larger of~$w_1$ and~$w_2$ as the left-hand side of the new rule. Assumptions on the ordering ensure that~$u \rightarrow^{\ast} v$ implies~$u \geq v$ and, since it is a well-ordering, the system $\mathcal{S}$ will be Noetherian. If $\mathcal{S}$ is complete, the irreducible elements correspond to the least representatives in the~$\leftrightarrow^{\ast}$ equivalence classes.

One of the most commonly used examples of a reduction ordering is a shortlex (or lenlex) ordering, discussed in Remark \ref{rem: len-lex ordering} and mainly used in this article. The Knuth-Bendix algorithm ensures the existence of a complete rewriting system $\mathcal{S'}$ which is equivalent to~$\mathcal{S}$.

\section{Rewriting systems in tRAAGs}\label{sec: rewriting_systems_tRAAGs}

For a group $G$ with a group presentation~$G = \Gpres{S}{R}$, its {\it monoid presentation} is given as~$G = \Mpres{S \cup S^{-1}}{R \cup I}$, with $S^{-1} = \{s^{-1} \mid s \in S\}$,~$I = \{ss^{-1} = s^{-1}s = 1 \mid s \in S\}$.

Let $\Gamma$ be a mixed graph defining $G = G_{\Gamma}$ as in Definition \ref{def_of_traags}. As usual,~$V$ denotes the vertices of~$\Gamma$, and $V^{-1}$ the inverses of elements in $V$. Assume $G$ is given via a monoid presentation with generators $S = V \cup V^{-1}$. We fix a total ordering of $S$, e.g. 
\[
v_1 < v_1^{-1} < \ldots < v_n < v_n^{-1},
\]
which induces a lexicographical ordering of words over $S$. We provide a rewriting system $\mathcal{R}_0$ for $G$ using lenlex ordering of rules as:
\begin{align}\label{eq: tRAAG rewriting system}
\mathcal{R}_0 = \{
& y^{\beta} x^{\alpha} \rightarrow x^{\alpha} y^{\beta} & & \hspace{-8cm} \text{ if } [x, y] \in E, \nonumber \\
& y^{\beta} x^{\alpha} \rightarrow x^{-\alpha} y^{\beta} & & \hspace{-8cm} \text{ if } [x, y \rangle \in E, \nonumber \\
& y^{\beta} x^{\alpha} \rightarrow x^{\alpha} y^{-\beta} & & \hspace{-8cm} \text{ if } \langle x, y] \in E, \nonumber \\
 & z^{\gamma} z^{-\gamma} \rightarrow 1, & \nonumber \\ 
& \text{where } x, y, z \in V \text{ with } x < y, \text{ and } \alpha, \beta, \gamma \in \{-1, 1\}\}.
\end{align}
The first three types of rules are called {\it edge rewrite-rules}, and the fourth type one is called a {\it vertex rewrite-rule}.

The rewriting system $\mathcal{R}_0$ is finite and Noetherian, but not necessarily confluent. We apply the Knuth-Bendix algorithm on $\mathcal{R}_0$ to obtain a complete (not necesarily finite) rewriting system $\mathcal{R}$, which is equivalent to $\mathcal{R}_0$. The rules of $\mathcal{R}$ will be described systematically in a uniform way in Subsection \ref{subsec: CRS for tRAAGs}.

It might happen that $\mathcal{R}$ is finite and complete, but not in general. An example of an infinite complete rewriting system is given in \cite{hermiller1995algorithms} for the pentagon RAAG. 


Moreover, each step of completion produces a rule coming from a critical pair between rules already obtained in a previous step.

Since $\mathcal{R}_0$ is terminating and there is a total ordering of the words, the algorithm ensures the existence of an equivalent rewriting system $\mathcal{R}$ which is complete.


\subsection{Some cases with a complete rewriting system}

Here we provide some cases of tRAAGs with a finite complete rewriting system (FCRS shortly), coming from the definition of $\mathcal{R}_0$.

\begin{example}\label{ex: FCRS for the Z2}
One has a FCRS for $\Z^2 = \Gpres{a,b}{a b = b a}$ given via:
\begin{align*}
\mathcal{R} = \{ & a^{\alpha} a^{-\alpha} \rightarrow 1, \quad b^{\beta} b^{-\beta} \rightarrow 1, \quad b^{\beta} a^{\alpha} \rightarrow a^{\alpha} b^{\beta}, \\
& \text{where } \alpha, \beta \in \{-1, 1\}\}.
\end{align*}
Indeed, the system is terminating, so from Lemma \ref{lem: locally confluent} we only need to check if critical pairs resolve. The only way we obtain critical pairs is overlapping vertex rules with edge rules, e.g.~$a^{\alpha} a^{-\alpha} \rightarrow 1$ with $b^{\beta} a^{\alpha} \rightarrow a^{\alpha} b^{\beta}$, and they resolve because:
\begin{align*}
b^{\beta}a^{\alpha}a^{-\alpha} & = b^{\beta}\underline{a^{\alpha}a^{-\alpha}} \rightarrow b^{\beta}, \text{ and} \\
b^{\beta}a^{\alpha}a^{-\alpha} & =
\underline{b^{\beta} a^{\alpha}}a^{-\alpha} \rightarrow a^{\alpha}\underline{b^{\beta}a^{-\alpha}} \rightarrow \underline{a^{\alpha}a^{-\alpha}}b^{\beta} \rightarrow b^{\beta}.
\end{align*}
Similar checking shows that all critical pairs resolve.
\end{example}

\begin{example}\label{ex: FCRS for the Klein bottle}
Analogous to the previous example, $K = \Gpres{a,b}{a b a = b}$ has a FCRS:
\begin{align*}
\mathcal{R} = \{ & a^{\alpha} a^{-\alpha} \rightarrow 1, \quad b^{\beta} b^{-\beta} \rightarrow 1, \quad b^{\beta} a^{\alpha} \rightarrow a^{-\alpha} b^{\beta}, \\
& \text{where } \alpha, \beta \in \{-1, 1\}\}.
\end{align*}
Here a slightly different critical pair is overlapping vertex rules with the (directed) edge rules, e.g. $a^{\alpha} a^{-\alpha} \rightarrow 1$ with $b^{\beta} a^{\alpha} \rightarrow a^{-\alpha} b^{\beta}$, and they resolve as one has:
\begin{align*}
b^{\beta}a^{\alpha}a^{-\alpha} & = b^{\beta}\underline{a^{\alpha}a^{-\alpha}} \rightarrow b^{\beta}, \text{ and} \\
b^{\beta}a^{\alpha}a^{-\alpha} & =
\underline{b^{\beta} a^{\alpha}}a^{-\alpha} \rightarrow a^{-\alpha}\underline{b^{\beta}a^{-\alpha}} \rightarrow \underline{a^{-\alpha}a^{-(-\alpha)}}b^{\beta} \rightarrow b^{\beta}.
\end{align*}
\end{example}


\subsubsection{Twisted RAAGs over complete graphs}
In Example \ref{ex: FCRS for the Z2} and Example \ref{ex: FCRS for the Klein bottle} we saw FCRS for $\Z^2$ and the Klein bottle group respectively. These examples belong to a more general result, which we give here by providing a FCRS for tRAAGs based on complete graphs.

Let $\Gamma = (V, E)$ be a complete mixed graph with $V = \{v_1, \dots, v_n\}$, and $G = \G$ the corresponding tRAAG. Set $v_1 < v_1^{-1} < \ldots < v_n < v_n^{-1}$ to be a total ordering of generators, 
and let~$\mathcal{R}_0$ be the rewriting system defined by Equation \eqref{eq: tRAAG rewriting system}.

Since $\mathcal{R}$ is defined by a length-lexicographical ordering, it is terminating. Moreover, the overlaps between vertex-rules and edge-rules in $\mathcal{R}$ resolve (see Example \ref{ex: FCRS for the Z2}). In this case we have additional critical pairs coming from two distinct edge rules; e.g. if $v_1 < v_2 < v_3$, and~$[v_1, v_2],\, [v_2, v_3 \rangle,\, [v_3, v_1 \rangle$ are edges, then we have a critical pair from the rules 

\[
v_3^{\alpha_3} v_2^{\alpha_2} \rightarrow v_2^{-\alpha_2} v_3^{\alpha_3} \text{ and } v_2^{\alpha_2} v_1^{\alpha_1} \rightarrow v_1^{\alpha_1} v_2^{\alpha_2} \text{ with each $\alpha_i$ taking values in $\{-1, 1\}$}.
\]
This critical pair resolves because:
\begin{align*}
v_3^{\alpha_3} v_2^{\alpha_2} v_1^{\alpha_1} & = \underline{v_3^{\alpha_3} v_2^{\alpha_2}}v_1^{\alpha_1} \rightarrow v_2^{-\alpha_2} \underline{v_3^{\alpha_3}v_1^{\alpha_1}} \rightarrow \underline{v_2^{-\alpha_2} v_1^{\alpha_1}}v_3^{-\alpha_3} \rightarrow v_1^{\alpha_1} v_2^{-\alpha_2}v_3^{-\alpha_3}, \text{ and}\\
v_3^{\alpha_3} v_2^{\alpha_2} v_1^{\alpha_1} & = v_3^{\alpha_3}  \underline{v_2^{\alpha_2}v_1^{\alpha_1}}\rightarrow \underline{v_3^{\alpha_3}  v_1^{\alpha_1}} v_2^{\alpha_2} \rightarrow v_1^{\alpha_1} \underline{v_3^{-\alpha_3}v_2^{\alpha_2}} \rightarrow v_1^{\alpha_1} v_2^{-\alpha_2}v_3^{-\alpha_3}.
\end{align*}
Similar calculations show that all critical pairs resolve. Hence, $\mathcal{R}$ is finite and complete.

\subsubsection{Twisted RAAGs over bipartite graphs}\label{subsec: Twisted RAAGs over bipartite graphs}
Here we give a FCRS for a tRAAG based on a bipartite graph.

Let $\Gamma = (V, E)$ be a bipartite mixed graph with $V = A \sqcup B$, where any edge in $E$ has one of its endpoints in $A$ and the other in $B$; let $\G$ be the corresponding tRAAG over $\Gamma$.

Set $A = \{a_1, \ldots, a_m\}$ and $B = \{b_1, \ldots, b_n\}$, and let 
\[
a_1 < a_1^{-1} < \cdots < a_m , a_m^{-1} < b_1 < b_1^{-1} < \cdots < b_n < b_n^{-1}
\]
be a total order in $V$. Let $\mathcal{R}$ be the rewriting system defined by the following set of rules:
\begin{align*}
\mathcal{R} = \{ & x^{\alpha} x^{-\alpha} \rightarrow 1, \\
& b_j^{\beta} a_i^{\alpha} \rightarrow a_i^{\alpha} b_j^{\beta} \text{ if } [a_i, b_j] \in E, \\
& b_j^{\beta} a_i^{\alpha} \rightarrow a_i^{-\alpha} b_j^{\beta} \text{ if } [a_i, b_j \rangle \in E, \\
& b_j^{\beta} a_i^{\alpha} \rightarrow a_i^{\alpha} b_j^{-\beta} \text{ if } \langle a_i, b_j] \in E, \\
& \text{where } x \in V \text{ and } \alpha, \beta \in \{-1, 1\}\}
\end{align*}

Since $\mathcal{R}$ is defined by a length-lexicographical ordering, it is terminating. Moreover, the only critical pairs between the rules in $\mathcal{R}$ come from vertex and edge rules, which resolve (as in Examples \ref{ex: FCRS for the Z2} and \ref{ex: FCRS for the Klein bottle}). Hence, $\mathcal{R}$ is finite and complete.

\begin{example}[Star-shaped graphs] This is an example of a bipartite graph, so it has a FCRS associated to it. This exaample serves more to illustrate what happens when we move one letter along a word, which affects the signs of powers of letters in the word, and the sign of the moving letter as well, which is illustrated in Equation \eqref{eq: star shape equation}.

The edges on the link of any vertex $x$ in $\Gamma$ can have three types (see Figure \ref{fig: three types of edges on the link}): directed edges facing away from $x$, directed edges facing towards $x$, and undirected edges.

\begin{figure}[H]
\centering
\begin{tikzpicture}[>={Straight Barb[length=7pt,width=4pt]}, rotate = 22.5] 

\draw [dotted,domain=55:80, thick] plot ({2*cos(\x)}, {2*sin(\x)});

\draw[fill=black] (0,0) circle (1.5pt) node[below] {};
\draw[fill=black] (-0.07,-0.07) circle (0pt) node[below] {$x$};

\draw[fill=black] ({2*cos(90)}, {2*sin(90)}) circle (1.5pt) node[above] {$a_1$};

\draw[fill=black] ({2*cos(45)}, {2*sin(45)}) circle (1.5pt) node[above] {$a_m$};

\draw[thick, ->] (0,0) -- ({2*cos(90)}, {2*sin(90)});
\draw[thick, ->] (0,0) -- ({2*cos(45)}, {2*sin(45)});

\draw [dotted,domain=295:320, thick] plot ({2*cos(\x)}, {2*sin(\x)});

\draw[fill=black] ({2*cos(330)}, {2*sin(330)}) circle (1.5pt) node[right] {$b_1$};

\draw[fill=black] ({2*cos(285)}, {2*sin(285)}) circle (1.5pt) node[below] {$b_n$};

\draw[thick, <-] (0,0) -- ({2*cos(330)}, {2*sin(330)});
\draw[thick, <-] (0,0) -- ({2*cos(285)}, {2*sin(285)});

\draw [dotted,domain=175:200, thick] plot ({2*cos(\x)}, {2*sin(\x)});

\draw[fill=black] ({2*cos(165)}, {2*sin(165)}) circle (1.5pt) node[left] {$c_p$};

\draw[fill=black] ({2*cos(210)}, {2*sin(210)}) circle (1.5pt) node[below] {$c_1$};

\draw[thick] (0,0) -- ({2*cos(165)}, {2*sin(165)});
\draw[thick] (0,0) -- ({2*cos(210)}, {2*sin(210)});

\end{tikzpicture}\caption{Three types of edges on the link of a vertex}\label{fig: three types of edges on the link}
\end{figure}
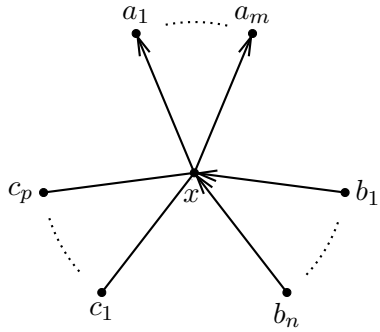

Set $\Link(x) = A \sqcup B \sqcup C$, where for any $a \in A,\, b \in B,\, c \in C$ one has that $[x, a \rangle$, $\langle x, b]$, and~$[x, c]$ are edges in $\Gamma$. So for any $\alpha, \beta, \gamma, \epsilon \in \{-1, 1\}$ one has:
\begin{equation}\label{eq: rewriting equations}
a^{\alpha} x^{\epsilon} = x^{-\epsilon} a^{\alpha}, \quad
b^{\beta} x^{\epsilon} = x^{\epsilon} b^{-\beta}, \quad
c^{\gamma} x^{\epsilon} = x^{\epsilon} c^{\gamma}.
\end{equation}

If $w = g_1^{j_1} \ldots g_k^{j_k}$ is a word over $A \sqcup B \sqcup C$, with each $j_i \in \{-1, 1\}$ for $1 \leq i \leq k$, then one can shuffle generators appearing in $w$ and $x^{\pm 1}$ to go from $wx^\epsilon$ to a word of the form $x^{\epsilon'} w'$, where both $\epsilon'$ and $w'$ are characterized with the following equation:
\begin{equation}\label{eq: star shape equation}
g_1^{j_1} \ldots g_k^{j_k} x^\epsilon = x^{\epsilon'} g_1^{j_1'} \ldots g_k^{j_k'},
\end{equation}
where $\epsilon' = (-1)^l \cdot \epsilon$ with
$l = |\{1 \leq i \leq k \mid g_i \in A \}|$, and one also has the following:
\[
    j_i' = 
\begin{cases}
    j_i, & \text{ if } g_i \in A \sqcup C, \\
    -j_i, & \text{ if } g_i \in B.
\end{cases}
\]
\end{example}

\begin{remark}
With respect to Equation \eqref{eq: rewriting equations} we will use the following notation:
\begin{equation}\label{eq: rewriting equations with triangle brackets}
\langle a, x \rangle = \langle x, a \rangle, \quad
\langle b, x \rangle = \langle x, b \rangle, \quad
\langle c, x \rangle = \langle x, c \rangle,
\end{equation}
which keeps track of the position of the letters (not the signs). However, once we know the position, we can rewrite the exact equality using the information on edges from the graph~$\Gamma$. Moreover, we will use the same notation when the rule comes from words, e.g. Equation \eqref{eq: star shape equation} can be written as $\langle g_1 \ldots g_k, x \rangle = \langle x,  g_1 \ldots g_k \rangle$, which embodies $2^{k+1}$ rules in itself, one for every tuple $(j_1, \ldots, j_k, \epsilon)$ where the coordinates belong to $\{-1, 1\}$. With a slight abuse of notation we will also write the above as $\langle g_1, \ldots , g_k, x \rangle = \langle x,  g_1,  \ldots , g_k \rangle$. 
\end{remark}


\subsection{CRS for tRAAGs}\label{subsec: CRS for tRAAGs}

In this subsection we characterize the rules of a complete rewriting system $\mathcal{R}$ coming via the Knuth-Bendix algorithm, applied on~$\mathcal{R}_0$ which is defined in Equation \eqref{eq: tRAAG rewriting system}. The work here follows the route of \cite{chouraqui2009rewriting} for RAAGs.

Denote by $\mathcal{R}_n$ the rewriting system obtained at the $n$-th step of the completion algorithm. At each step we reduce words in critical pairs to their normal forms modulo the rewriting system at the earlier step. This way we focus only on critical overlap pairs, and avoid inclusion critical pairs.

To a word $w = g_1 \ldots g_k$ with $g_j \in \{v_j, v_j^{-1}\}$ we associate a set of words
\[
\langle w \rangle = \{v_1^{\alpha_1} \ldots v_k^{\alpha_k} \mid \alpha_i \in \{-1, 1\} \text{ for } 1 \leq i \leq k\}.
\]
For the set of words $W = \langle w \rangle$ define its prefixes to be 
\[
\pref(W) = \{v_1^{\alpha_1}, v_1^{\alpha_1} v_2^{\alpha_2},  \ldots,  v_1^{\alpha_1} v_2^{\alpha_2} \cdots v_k^{\alpha_k} \mid \alpha_i \in \{-1, 1\} \text{ for } 1 \leq i \leq k\}.
\]

\begin{lemma}\label{lem: on prefix rules and R_0}
Assume the following set of rules occur in $\mathcal{R}$:
\begin{align*}
\langle v_1, x \rangle & \rightarrow \langle x, v_1 \rangle \\
\langle x, v_i \rangle & \rightarrow \langle v_i, x \rangle \text{ for } 1 \leq i \leq k,
\end{align*}

where $x, v_i \in V$ for $1 \leq i \leq k$ and such that $\{x, v_1, v_2, \ldots, v_k\}$ is not a clique.

Suppose that the rules $\langle w, x \rangle \rightarrow \langle x, w \rangle$ appear in $\mathcal{R}$, with $w = v_1 v_{i_2} \ldots v_{i_k}$. Then for every prefix~$u \in \pref(\langle w \rangle)$, the rules $\langle u, x \rangle \rightarrow \langle x, u \rangle$ also occur in $\mathcal{R}$; these rules are obtained at a step earlier than~$\langle w, x \rangle \rightarrow \langle x, w \rangle$.

\end{lemma}


Let $x, g_i \in V \cup V^{-1}$ and $u = g_1 g_{2} \ldots g_{k}$. We say that the rules $\langle u, x \rangle \rightarrow \langle x, u \rangle$ satisfy the {\it prefix condition} if the rules $\langle g_1, x \rangle \rightarrow \langle x, g_1 \rangle$ and $\langle x, g_i \rangle \rightarrow \langle g_i, x \rangle$  for all $2 \leq i \leq k$, appear in~$\mathcal{R}_0$ as well.

\begin{lemma}
Assume $T = \{v_1, \ldots, v_k\}$ is a clique. Then any word $w \in (T \cup T^{-1})^{\ast}$ can be reduced to its normal form modulo $\mathcal{R}$ by using only rules from $\mathcal{R}_0$. 
\end{lemma}
\begin{proof}
The overlaps between rules involving only elements from a clique, do not create other rules and the proof is basically the same as for tRAAGs over complete graphs: fix a total ordering on the generators; such a word $w$ reduces to a unique irreducible word, which is the least one lexicographically that is obtained by applying rules from $\mathcal{R}_0$.
\end{proof}
\begin{proof}[Proof of Lemma \ref{lem: on prefix rules and R_0}]
We proceed by induction on the number of steps in the Knuth-Bendix algorithm and we prove also the following:
\begin{itemize}[
  align=left,
  leftmargin=8em,
  itemindent=0pt,
  labelsep=0pt,
  labelwidth=8em
]
\raggedright
\item[{\bf star condition}:] if $x$ is the last letter in the left-hand side of a rule that shuffles with all the letters appearing before, then in the right-hand side $x$ is the first letter.
\end{itemize}
At step $0$ of the algorithm, there is nothing to prove and the assumptions hold trivially.

Now on the first step of completion, one has an overlap between the rules $\langle v_1, x \rangle \rightarrow \langle x, v_1 \rangle$,  and $\langle x, v_{i_2} \rangle \rightarrow \langle v_{i_2}, x \rangle$ in $\mathcal{R}_0$ with $i_2 \neq 1$, which gives
\begin{center}
\begin{tikzpicture}[>={Straight Barb[length=6pt,width=5pt]}, inner sep=2pt]
    \node at (0,0) (u) {$\langle v_1, x, v_{i_2} \rangle \,$}; 
    \node at (3,0.5) (v1) {$\, \langle x, v_1, v_{i_2} \rangle$}; 
    \node at (3,-0.5) (v2) {$\, \langle v_1, v_{i_2}, x \rangle$}; 
\path[->] (u) edge  (v1);
\path[->] (u) edge  (v2);
\end{tikzpicture}
\end{center}

This produces a new rule $\langle v_1 , v_{i_2}, x  \rangle \rightarrow \langle x, v_1, v_{i_2} \rangle$, and star condition is fulfilled as well. Now $\rr{x, v_1, v_{i_2}}$ are irreducible modulo $\mathcal{R}_0$ because they are prefixes of $\rr{x, v_1 v_{i_2}\ldots v_{i_k}}$.
Also the words $\rr{v_1, v_{i_2}, x}$ are irreducible modulo $\mathcal{R}_0$, as all subwords are irreducible modulo $\mathcal{R}_0$: indeed, $\rr{v_1, v_{i_2}}$ is irreducible modulo $\mathcal{R}_0$ because on the contrary, there would be an inclusion critical pair with the rule $\langle v_1 v_{i_2} \dots v_{i_k}, x \rangle \rightarrow \langle x, v_1 v_{i_2} \dots v_{i_k} \rangle$  in $\mathcal{R}$, which is not possible by our assumption. 

Suppose that we obtained the rule $\langle v_1 v_{i_2} \dots v_{i_{n-2}}, x \rangle \rightarrow \langle x, v_1 v_{i_2} \dots v_{i_{n-2}} \rangle$ at step $(m-1)$ of completion of $\mathcal{R}_0$, where $v_{i_j} \in \{v_2, \ldots, v_k\}$. Then at step $m$, the overlap between this rule and $ \langle x, v_{i_{n-1}} \rangle \rightarrow \langle v_{i_{n-1}}, x \rangle$, with $i_{n-1} \neq 1$, yields:

\begin{center}
\begin{tikzpicture}[>={Straight Barb[length=6pt,width=5pt]}, inner sep=2pt]
    \node at (0,0) (u) {$\langle v_1 v_{i_2} \dots v_{i_{n-2}}, x, v_{i_{n-1}} \rangle \,$}; 
    \node at (5,0.5) (v1) {$\, \langle x, v_1 v_{i_2} \dots v_{i_{n-1}} \rangle$}; 
    \node at (5,-0.5) (v2) {$\, \langle v_1 v_{i_2} \dots v_{i_{n-1}}, x \rangle$}; 
\path[->] (u) edge  (v1);
\path[->] (u) edge  (v2);
\end{tikzpicture}
\end{center}

It remains to show that one obtains a new rule, whose left side is $\langle v_1 v_{i_2} \dots v_{i_{n-1}}, x \rangle$, and~$\langle x, v_1 v_{i_2} \dots v_{i_{n-1}} \rangle$ as the right-hand side; for this it suffices to demonstrate that these are both irreducible modulo the rewriting system $\mathcal{R}_{m-1}$ obtained one step earlier.
The side $\langle x, v_1 v_{i_2} \dots v_{i_{n-1}} \rangle$ is irreducible modulo $\mathcal{R}$ because it is a prefix of $\langle x v_1 v_{i_2} \dots v_{i_{n-1}} v_{i_n} \rangle$, which is irreducible modulo $\mathcal{R}$. On the other hand, we prove that $\rr{v_1 v_{i_2} \dots v_{i_{n-1}}x}$ is irreducible modulo $\mathcal{R}_{m-1}$, by showing that each of its subwords cannot appear as a left-hand side of a rule in $\mathcal{R}_{m-1}$. Suppose there is such a rule with left side $\langle v_{i_j} v_{i_{j+1}}\dots v_{i_{n-1}}x \rangle$, where $i_j \neq 1$. From star condition, we have that $\langle v_{i_j} v_{i_{j+1}}\dots v_{i_{n-1}}x \rangle \rightarrow \langle x, \ldots \rangle$ since $x$ shuffles with all the~$v_{i_k}$ and this would imply that $v_{i_j} > x$. However, the existence of the rule $\langle x, v_{i_j} \rangle \rightarrow \langle v_{i_j}, x \rangle$ in~$\mathcal{R}_0$, implies~$x > v_{i_j}$, providing a contradiction. Hence, there can be no such rule and analogously one checks the same for other subwords as well as we would obtain inclusion critical pairs with~$\langle x v_1 v_{i_2} \dots v_{i_{n-1}} v_{i_n} \rangle$, so $\langle v_1 v_{i_2} \dots v_{i_{n-1}}x \rangle$ is irreducible modulo~$\mathcal{R}_{m-1}$.

Ultimately, there is a rule 
$\langle v_1 v_{i_2} \dots v_{i_{n-1}}x \rangle \rightarrow \langle x v_1 v_{i_2} \dots v_{i_{n-1}} \rangle$ in $\mathcal{R}$ and star condition holds as well, as required.

\end{proof}
Next, we show that overlaps in $\mathcal{R}$ belong to $S = V \cup V^{-1}$.
\begin{lemma}
Assume the rules 
\begin{align*}
\langle x, v_1, v_{2}, \ldots v_{k} \rangle & \rightarrow \langle v_{k}, x, v_1, v_{2}, \ldots, v_{k-1} \rangle \\
\langle v_1, v_{2}, \ldots v_{k}, y \rangle & \rightarrow \langle y, v_1, v_{2}, \ldots, v_{k} \rangle
\end{align*}
are in $\mathcal{R}$ and  satisfy the prefix condition. Then $k = 1$.
\end{lemma}

\begin{proof}
Assume that $k > 1$. The first written rule satisfies the condition on prefixes, 
so the rule~$\rr{v_k, v_1} \rightarrow \rr{v_1, v_k}$ occurs in $\mathcal{R}_0$, 
in particular implying that $v_k > v_1$. Analagously, from the second rule one has $v_1 > y$ and $y > v_k$. Ultimately, $v_1 > v_k$, which is a contradiction.
\end{proof}

\begin{lemma}
The Knuth-Bendix algorithm of completion of $\mathcal{R}_0$ yields an equivalent complete rewriting system $\mathcal{R}$ satisfying these conditions:
\begin{itemize}
\item[(i)] All the rules in $\mathcal{R}$ appear as $\langle u, x \rangle \rightarrow \langle x, u \rangle$ with $u = \langle v_1 v_2 \dots v_{k-1} v_k \rangle$ and $x, v_i \in V \cup V^{-1}$ for all $1 \leq i \leq k$, and they satisfy the prefix condition.
\item[(ii 
)] New rules are created only from overlaps between rules from $\mathcal{R}$ at left and rules from~$\mathcal{R}_0$ at right; all the other kinds of overlaps resolve (including overlaps between rules from~$\mathcal{R}_0$ at left and rules from~$\mathcal{R}$ at right).

\item[(iii)] The rules created in the $n$-th step of completion, have words of length $n+2$ in each side.
\end{itemize}
\end{lemma}

\begin{proof}
We use induction on the number of steps of the Knuth-Bendix algorithm on $\mathcal{R}_0$.

In step $0$, the rules in $\mathcal{R}_0$ are of the form $\langle v, x \rangle \rightarrow \langle x,v \rangle$, with $\{x, v\} \in E$. Part (i) holds trivially. Also, all rules have words of length $2$ in each side, so (iii) also holds.

In step $1$, there are overlaps between $\langle z, y \rangle \rightarrow \langle y, z \rangle$ and
$\langle y, x \rangle \rightarrow \langle x, y \rangle$; therefore, new rules~$\langle z, x, y \rangle \rightarrow \langle y, z, x \rangle $ are produced. Hence, both (i) and (iii) hold.

Now we prove that part (ii) holds as well in step $1$. Suppose that in $\mathcal{R}_0$ we have the rules~$\langle y, z_1 \rangle \rightarrow \langle z_1, y \rangle$ and
$\langle z_1, x_1 \rangle \rightarrow \langle x_1, z_1 \rangle$ which produce the rules $\langle y, x_1, z_1 \rangle \rightarrow \langle z_1, y, x_1 \rangle$.

It remains to show that we obtain a resolution between overlaps of $\langle z, x, y \rangle \rightarrow \langle y, z, x \rangle$ on the left, with $\langle y, x_1, z_1 \rangle \rightarrow \langle z_1, y, x_1 \rangle$ on the right. 

Firstly, the rule $\langle z, x, y \rangle \rightarrow \langle y, z, x \rangle$ overlaps with $\langle y, z_1 \rangle \rightarrow \langle z_1, y \rangle$ to give: 

\begin{center}
\begin{tikzpicture}[>={Straight Barb[length=6pt,width=5pt]}, inner sep=2pt]
    \node at (0,0) (u) {$\langle z, x, y, z_1 \rangle$}; 
    \node at (3,0.5) (v1) {$\langle y, z, x, z_1 \rangle$}; 
    \node at (3,-0.5) (v2) {$\langle z, x, z_1, y \rangle$}; 
    \path[->] (u) edge  (v1);
 	\path[->] (u) edge  (v2);
\end{tikzpicture}
\end{center}
which provides the rule: $\langle z, x, z_1, y \rangle \rightarrow \langle y, z, x, z_1 \rangle$. Together, with $\langle z_1, x_1 \rangle \rightarrow \langle x_1, z_1 \rangle$ this gives a resolution to:
\begin{center}
\begin{tikzpicture}[>={Straight Barb[length=6pt,width=5pt]}, inner sep=2pt]
    \node at (0,0) (u) {$\langle z, x, y, x_1, z_1 \rangle$}; 
    \node at (5,0.55) (v1) {$\langle y, z, x, x_1, z_1 \rangle$}; 
    \node at (5,-0.55) (v2) {$\langle z, x, z_1, y, x_1 \rangle$}; 
    \path[->] (u) edge  (v1);
 	\path[->] (u) edge  (v2);
 	\path[->, dotted, thick] (v2) edge  (v1);
\end{tikzpicture}
\end{center}
Hence, (ii) holds in the first step.

Assume that the Knuth-Bendix algorithm fulfills conditions (i), (ii), (iii) in its $(n-1)$-th step. Now, in the $n$-th step, new rules are obtained from overlap critical pairs between rules obtained at the $(n-1)$-th step and rules obtained at the $(n-1)$-th step or earlier. Suppose~$\langle u, x \rangle \rightarrow \langle x, u \rangle $ is a rule obtained at the $(n-1)$-th step which satisfies the prefix condition, with $u = \langle v_1 v_2 \dots v_{n-1} v_n \rangle$ and $x, v_i \in V \cup V^{-1}$ for all $1 \leq i \leq n$. Moreover, assume the rule $\langle x, w, y \rangle \rightarrow \langle y, x, w \rangle$ in $\mathcal{R}_{n-1}$ satisfies the prefix condition, where $w$ is a word and~$y \in V \cup V^{-1}$.
If $w$ is not the empty word, then (ii) implies that the overlap between rules~$\langle  u, x \rangle \rightarrow \langle x, u \rangle$ and $\langle  x, w, y \rangle \rightarrow \langle y, x, w \rangle$ resolves. Hence, assume $w$ is the empty word, i.e. consider the overlaps between $\langle  u, x \rangle \rightarrow \langle x, u \rangle$ and $\langle  x, y \rangle \rightarrow \langle y, x \rangle$:
\begin{center}
\begin{tikzpicture}[>={Straight Barb[length=6pt,width=5pt]}, inner sep=2pt]
    \node at (0,0) (u) {$\langle u, x, y \rangle $}; 
    \node at (3,0.5) (v1) {$\langle x, u, y \rangle $}; 
    \node at (3,-0.5) (v2) {$\langle u, y, x \rangle $}; 
    \path[->] (u) edge  (v1);
 	\path[->] (u) edge  (v2);
\end{tikzpicture}
\end{center}
We distinguish three cases to check the critical overlaps above. We prove first that conditions~(i) and~(iii) are satisfied in each case, and then later deal with case (ii).

{\bf Case 1:} Both $\langle x,u,y \rangle$ and $\langle  u,y,x \rangle$ are freely reduced and irreducible modulo $\mathcal{R}_{n-1}$. A new rule $\langle u,y,x \rangle \rightarrow \langle x,u,y \rangle$ is created which satisfies the prefix condition, so assumption (i) holds in $\mathcal{R}_n$. By the induction assumption (iii) we know that both $\langle u, y, x \rangle$ and $\langle x,u,y \rangle$ have length $n+2$, so (iii) holds as well.

{\bf Case 2:} Both $\langle x,u,y \rangle$, $\langle u,y,x \rangle$ are freely reduced but potentially reducible modulo~$\mathcal{R}_{n-1}$.

If $\langle u, y \rangle \rightarrow \langle y, u \rangle$ occurs in $\mathcal{R}_{n-1}$ then the overlap resolves by first using~$\langle u,y \rangle \rightarrow \langle y,u \rangle$ and then proceeding by~$\langle u,x \rangle \rightarrow \langle x,u \rangle$ and $\langle x, y \rangle \rightarrow \langle y,x \rangle$.

If one has a split $u=u_1u_2$, with $u_1$ non-empty, and $\langle u_2,y \rangle \rightarrow \langle y,u_2 \rangle$ takes place in $\mathcal{R}_{n-1}$, then one has:
\begin{center}
\begin{tikzpicture}[>={Straight Barb[length=6pt,width=5pt]}, inner sep=2pt]
    \node at (0,0) (u) {$\langle u, x, y \rangle $}; 
    \node at (0,-0.55) (u1) {$\verteq$}; 
    \node at (0,-1) (u2) {$\langle u_1, u_2, x, y \rangle$}; 
    \node at (3,-1) (v1) {$\langle u_1, u_2, y, x \rangle$}; 
    \node at (6,-1) (v2) {$\langle u_1, y, u_2, x \rangle$}; 
    \node at (3,0) (w1) {$\langle x, u_1, u_2, y \rangle$};
    \node at (6,0) (w2) {$\langle x, u_1, y, u_2 \rangle$};
     
    \path[->] (u2) edge  (v1);
    \path[->] (v1) edge  (v2);
 	\path[->] (u)  edge  (w1);
 	\path[->] (w1) edge  (w2);
\end{tikzpicture}
\end{center}
Hence, a new rule: $\langle u_1, y, u_2, x \rangle \rightarrow \langle x, u_1, y, u_2 \rangle$ is created, that satisfies the prefix condition, therefore both (i) and (iii) are fulfilled in this case as well.

{\bf Case 3:} Assume that at least one of $\langle x,u,y \rangle$ and $\langle u,y,x \rangle$ is not freely reduced. Since all the subwords $\langle x,u \rangle$, $u$, and $\langle y,x \rangle$ are reduced, one has that $\langle u,y \rangle$ is not freely reduced, i.e.~$v_k = y^{-1}$. One has:
\begin{center}
\begin{tikzpicture}[>={Straight Barb[length=6pt,width=5pt]}, inner sep=2pt]
    \node at (0,0) (u) {$\langle u, x, y \rangle $}; 
    \node at (0,-0.55) (u1) {$\verteq$}; 
    \node at (0,-1) (u2) {$\langle v_1 v_2 \dots v_{k-1} y^{-1}, x, y \rangle$}; 
    \node at (5,-1) (v1) {$\langle v_1 v_2 \dots v_{k-1} y^{-1}, y, x  \rangle$}; 
    \node at (9.5,-1) (v2) {$\langle v_1 v_2 \dots v_{k-1}, x \rangle$}; 
    \node at (5,0) (w1) {$\langle x,v_1 v_2 \dots v_{k-1} y^{-1}, y \rangle$};
    \node at (9.5,0) (w2) {$\langle x,v_1 v_2 \dots v_{k-1} \rangle$};
     
    \path[->] (u2) edge  (v1);
    \path[->] (v1) edge  (v2);
 	\path[->] (u)  edge  (w1);
 	\path[->] (w1) edge  (w2);
\end{tikzpicture}
\end{center}
Note that when $x$ shuffles with $y$, only $x$ can change the sign of its power.

Since $\langle v_1 v_2 \dots v_{k-1}  \rangle$ belongs to $\pref(\langle u \rangle)$ and the rule $\langle u, x  \rangle \rightarrow \langle x, u \rangle$ in $\mathcal{R}_{n-1}$ satisfies the prefix condition, there is a rule
$\langle v_1 v_2 \dots v_{k-1}, x \rangle \rightarrow \langle x, v_1 v_2 \dots v_{k-1} \rangle$ obtained at an earlier step, which means that this  overlap resolves via this rule.
Ultimately, both (i) and (iii) hold in $\mathcal{R}$. Now, we proceed to show that the condition (ii) holds as well.

Suppose $\langle u, x \rangle \rightarrow \langle x, u \rangle$ is in $\mathcal{R}_n$, and satisfies the prefix condition. Let~$\langle x, w, y \rangle \rightarrow  \langle y, x, w \rangle$ be a rule in $\mathcal{R}_n \setminus \mathcal{R}_0$, where $w = \langle  v_1 v_2 \dots v_k  \rangle$ with $v_i \in V \cup V^{-1}$, such that there is an overlap with $\langle u, x  \rangle \rightarrow \langle x, u \rangle$ on the left and $\langle x, w, y \rangle \rightarrow \langle y, x, w \rangle$ on the right. To show that (ii) holds, we need to conclude that this overlap resolves.

From (i) we know that the rule $\langle x, w, y \rangle \rightarrow \langle y, x, w \rangle$ satisfies the prefix condition. So:
\begin{center}
\begin{tikzpicture}[>={Straight Barb[length=6pt,width=5pt]}, inner sep=2pt]
    \node at (0,0) (u) {$\langle u, x, w, y \rangle$}; 
    \node at (0,-0.55) (u1) {$\verteq$}; 
    \node at (0,-1) (u2) {$\langle u, x, v_1 v_2 \dots v_k, y \rangle$}; 
    \node at (5,-1) (v1) {$\langle u, y, x, w \rangle$}; 
    \node at (8.75,-1) (v2) {$\langle u, y, x, v_1 v_2 \dots v_k \rangle$}; 
    \node at (5,0) (w1) {$\langle x, u, w, y \rangle $};
    \node at (8.75,0) (w2) {$\langle x,u, v_1 v_2 \dots v_{k} y  \rangle$};
    
    \path[->] (u2) edge  (v1);
    \node at (6.5, 0) {$=$}; 
 	\path[->] (u)  edge  (w1);
 	\node at (6.5, -1) {$=$}; 
\end{tikzpicture}
\end{center}
As $\langle x, w, y \rangle \rightarrow \langle y, x, w \rangle$ satisfies the prefix condition, there are rules $\langle x, y \rangle \rightarrow \langle y, x \rangle$ and also~$\langle y, v_i \rangle \rightarrow \langle v_i, y \rangle$ in $\mathcal{R}_0$. Therefore, one has an overlap of $\langle u, x \rangle \rightarrow \langle x, u \rangle$ on the left and~$\langle x, y \rangle \rightarrow \langle y, x \rangle$ on the right, which created the rule $\langle u, y, x \rangle \rightarrow \langle x, u, y \rangle$. 

Hence, using this rule and the rules $\langle y, s_i  \rangle \rightarrow \langle s_i, y \rangle$ with $1 \leq i \leq k$ we obtain:
\begin{center}
\begin{tikzpicture}[>={Straight Barb[length=6pt,width=5pt]}, inner sep=2pt]
    \node at (0,0) (u) {$\langle u, x, w, y \rangle$}; 
    \node at (0,-0.55) (u1) {$\verteq$}; 
    \node at (0,-1) (u2) {$\langle u, x, v_1 v_2 \dots v_k, y \rangle$}; 
    \node at (4,-1) (v1) {$\langle u, y, x, v_1 v_2 \dots v_k \rangle$}; 
    \node at (8,-1) (v2) {$\langle x, u, y, v_1 v_2 \dots v_k \rangle$}; 
        \node at (12,-1) (v3) {$\langle x, u, v_1 v_2 \dots v_k, y \rangle$}; 
    \node at (4,0) (w1) {$\langle x, u, w, y \rangle $};

    \path[->] (u2) edge  (v1);
    \path[->] (v1) edge  (v2);
    \path[->] (v2) edge  (v3);
 	\path[->] (u)  edge  (w1);
 	
\end{tikzpicture}
\end{center}
Which means that this overlap resolves. 

Finally, we prove that the overlaps of rules with rules of free reduction resolve. Pick a rule~$\langle u, x \rangle \rightarrow \langle x, u \rangle$ satisfying the prefix condition, with $u = \langle v_1 v_2 \dots v_{k-1} v_k\rangle$ and with elements~$x, v_i \in X \cup X^{-1}$ for all $1 \leq i \leq k$. We first consider the overlap with $\langle u, x \rangle \rightarrow \langle x, u \rangle$ on the left and $\langle x,x^{-1}\rangle \rightarrow 1$ on the right:
\[
\langle uxx^{-1} \rangle \rightarrow \langle xux^{-1} \rangle \rightarrow  \langle xx^{-1}u \rangle \rightarrow u,
\]
as desired.

Similarly, the overlap with $\langle v_1^{-1}v_1 \rangle \rightarrow 1$ on the left and $\langle u, x \rangle \rightarrow \langle x, u \rangle$ on the right:
\begin{align*}
\langle v_1^{-1},v_1 v_2 \dots v_{k} x  \rangle \rightarrow \langle v_1^{-1},x, v_1 v_2 \dots v_{k} \rangle \rightarrow \langle x,v_1^{-1},v_1 v_2 \dots v_{k}  \rangle & \rightarrow \langle x,v_2 \dots v_{k} \rangle \rightarrow \ldots  \\ 
& \rightarrow \langle v_2 \dots v_{k} x \rangle,
\end{align*}
as desired. Therefore, condition (ii) holds as well in $\mathcal{R}$.
\end{proof}

\section{Normal form theorem in tRAAGs}\label{sec: normal_form}
When we talk about languages, words, and growth in tRAAGs, we will use $S = V\sqcup V^{-1}$ as our alphabet (or preferred monoid generating set in the context of groups), where~$V$ will always denote the vertex set of the defining graph $\Gamma$. 

The following relations hold for a tRAAG based on $\Gamma$:
\begin{enumerate}[itemsep=4pt,parsep=0pt,topsep=4pt, itemindent=40pt]
\item[(1)] $u^m u^n  = u^{m+n}$ for any $m,n \in \Z$, and any $u\in V\ga$,
\item[(2)] $u^m v^n  = v^n u^m$ for any $m,n \in \Z$ if $[u,v] \in \overline{E\ga}$,
\item[(3)] $u^m v^n = v^n u^{(-1)^n m}$ for any $m,n \in \Z$ if $[u,v\rangle \in \overrightarrow{E\ga}$.
\end{enumerate}
One obtains back the original relations by putting $m=n=1$.
The relation in (1) is called a {\it multiplication relation} in $\langle u \rangle$, while (2) and (3) are called {\it graph relations} based on $u$ (or in  $v$).

\begin{mydef}\label{def_word_on_vertex_group}
Any element $g\in G_{\Gamma}$ can be represented as a product $w = g_1\ldots g_r$ where each $g_i$ is an element of some $\langle u \rangle$ with $u\in V$; in this case we say that $w$ is a {\it word} representing the element $g$, and $g_i$ are {\it syllables} of $w$. The number of syllables in the word $w$ is called the {\it syllable length} of $w$, and denoted by $\text{sl}(w)$.  
\end{mydef}

\begin{mydef}\label{def_normal_form}
The word $w$ is called a {\it normal form} for $g \in G_{\ga}$ if $w$ represents $g$ in $G_{\ga}$ and~$w$ is reducible with respect to the CRS for $\G$.
\end{mydef}

\begin{remark}
Normal forms are products of syllables with the smallest syllable length in their corresponding class, also with the smallest shortlex length.
\end{remark}

In a group $G$ with a monoid generating set $S$ the {\it word problem} consists of determining weather or not a given word $w$ over $S$ represents the identity $1 \in G$.

As a corollary of the normal form theorem for tRAAGs we get the following:

\begin{cor}
In the class of tRAAGs the word problem is solvable.
\end{cor}

\begin{proof}
Let $w$ be a word in $G_{\Gamma}$. Then $w = 1$ in $G_{\Gamma}$ if and only if the normal form of $w$ is the empty sequence.
\end{proof}

\pgfdeclaredecoration{free hand}{start}
{
  \state{start}[width = +0pt,
                next state=step,
                persistent precomputation = \pgfdecoratepathhascornerstrue]{}
  \state{step}[auto end on length    = 3pt,
               auto corner on length = 3pt,               
               width=+2pt]
  {
    \pgfpathlineto{
      \pgfpointadd
      {\pgfpoint{2pt}{0pt}}
      {\pgfpoint{rand*0.15pt}{rand*0.15pt}}
    }
  }
  \state{final}
  {}
}
 \tikzset{free hand/.style={
    decorate,
    decoration={free hand}
    }
 } 
\def\freedraw#1;{\draw[free hand] #1;}

\section{Applications of the normal form theorem}\label{applications_normal_form}
    
Let $G = G_{\Gamma}$ be the tRAAG based on a mixed graph $\Gamma = (V, E)$, and $S = V \sqcup V^{-1}$.
\begin{notation}
Denote by $l(w)$ the length of a word $w$ with respect to $S$; and by $|g|$ the length of the shortest word representing $g$. One refers to $|g|$ as the length of $g$ in $G$.
\end{notation}
Forgetting about directions of edge of a mixed graph, gives a useful way on the study of similarities and differences between RAAGs and tRAAGs.
\begin{mydef}
Let $\Gamma = (V,E,D,o,t)$ be a mixed graph and $G_{\Gamma}$ the tRAAG based on~$\Gamma$. The underlying simplicial graph~$\overline{\Gamma} = (V,E)$ is called {\it the underlying graph of ${\Gamma}$}. Similarly, the right-angled Artin group $G_{\overline{\Gamma}}$ is called the {\it underlying RAAG} of $G_{\Gamma}$.
\end{mydef}

\begin{example}
For the fundamental group of the Klein bottle~$K = \langle a,b \mid aba = b\rangle$, its underlying RAAG is~$\Z^2$.
\end{example}
\begin{remark}
The correspondence provided above is quite natural, as a lot of geometric properties of $G_{\Gamma}$, agree with the corresponding properties of $G_{\overline{\Gamma}}$.
\end{remark}

Despite having similar presentations, one can notice major differences between tRAAGs and RAAGs. For example, one can have torsion in tRAAGs (see Section \ref{torsion_traags}), which is not the case in RAAGs. Also, one can have isomorphic tRAAGs based on non-isomorphic mixed graphs, as in the example below.

\begin{example}[see \cite{phdthesis}]\label{isomorphic_traags_but_not_graphs}
Let $\Gamma_1$ and $\Gamma_2$ be mixed graphs given as below:
\begin{figure}[H]
\centering
\begin{tikzpicture}[>={Straight Barb[length=7pt,width=6pt]},thick]
\draw[] (-1.5, 0) node[left] {$\Gamma_1 = $};
\draw[fill=black] (-1,0) circle (1pt) node[below] {$a$};
\draw[fill=black] (.5,0) circle (1pt) node[below] {$b$};
\draw[fill=black] (2,0) circle (1pt) node[below] {$c$};

\draw[] (4.25, 0) node[left] {,\;\;\;\;\;\; $\Gamma_2 = $};
\draw[fill=black] (5,0) circle (1pt) node[below] {$x$};
\draw[fill=black] (6.5,0) circle (1pt) node[below] {$y$};
\draw[fill=black] (8,0) circle (1pt) node[below] {$z$};

\draw[thick] (-1,0) -- (.5,0);
\draw[thick, ->] (.5,0) -- (2,0);
\draw[thick, ->] (6.5,0) -- (5,0);
\draw[thick, ->] (6.5,0) -- (8,0);
\end{tikzpicture}
\end{figure}
Associated to each $\Gamma_i$ we get a tRAAG, denoted by $G_i = G_{\Gamma_i}$, and presented as:
\begin{align*}
G_1 & = \langle a,b,c \mid a b = b a, \; b c = c b ^{-1}\rangle, \\
G_2 & = \langle x,y,z \mid y x = x y ^{-1} , \; y z = z y ^{-1} \rangle.
\end{align*}
Consider the maps $f\colon G_1 \to G_2$ and $g\colon G_2 \to G_1$, defined by $f(a) = x z ^{-1}$, $f(b) = y$, $f(c) = z$; and $g(x) = a c, g(y) = b,  g(z) = c$. Both of them are well-defined morphisms of groups as they respect the group relations; moreover $g\circ f = 1_{G_1}$ and $f\circ g = 1_{G_2}$ which means that~$G_1 \simeq G_2$. 
\end{example}


\subsection{Abelianization in tRAAGs}
Let $G$ be a group.
The abelianization of $G$ is 
the group $G^{ab} = G/[G,G]$, where $[G,G]$ is the commutator subgroup (generated by all the commutators $ghg^{-1}h^{-1}$ with $g,h \in G$).

In RAAGs the abelianization depends only on the number of vertices on the graph. Indeed, let $\Gamma$ be a graph with $|\Gamma| = n$, and $G_{\Gamma}$ the RAAG based on it. Then~$G_{\Gamma}^{ab} = \mathbb{Z}^n.$\\
In tRAAGs the situation is a bit different, and it depends on the direction of the edges as well. To get a better grasp at the situation let us discuss morphisms onto $\Z$.

\begin{prop}\label{prop: indicable tRAAGs}
There is a surjective morphism $\varphi\colon G_{\Gamma}\longrightarrow \mathbb{Z}$ 
if and only if there is a vertex $v \in V\Gamma$ which is not the origin of any directed edge.
\end{prop}
\begin{proof}
If such $v$ exists, we can define a map from $G_{\Gamma}$ to $\mathbb{Z}$ by sending $v$ to $1$ and all the other vertices to $0$. This map respects the relations on all the generators and therefore induces a surjective homomorphism. So the group maps onto $Z$.\\
On the other hand, by contradiction, assume that any vertex of our graph is the origin of at least one directed edge. Assume that there is a morphism $\varphi:G \longrightarrow \mathbb{Z}$. Pick any vertex $a$ in $\Gamma$, then by assumption there is an oriented edge $e = [a,b\rangle$ for some $b$. The relation coming from edge $e$ is:
$aba = b$, and this implies $2\varphi(a) = 0$. This means that~$\varphi$ is trivial (all generators are mapped to 0) and thus the group does not map onto $Z$. 
\end{proof}
Now we go back to the abelianization of tRAAGs. Let $\Gamma$ be a mixed graph, and $G$ the tRAAG based on $\Gamma$. Then $G^{ab}$ can be computed by adding the relations $xy = yx$ to the presentation of $G$ for all $x,y \in V\Gamma$. If $[a, b \rangle$ is a directed edge, we have the relation~$abab^{-1} = 1$. Since the relation~$aba^{-1}b^{-1} = 1$ is added to the presentation as well, we get $a^2 = 1$. This holds for any vertex $a$ when it is the origin of a directed edge. If we denote by $V_o$ the set of vertices that are the origin of a directed edge, then
$$G^{ab} = \mathbb{Z}^{|V\Gamma - V_o|}\times \mathbb{Z}_2^{|V_o|}.$$

Since isomorphic groups have isomorphic abelianizations we have the following
\begin{cor}
Let $\Gamma_1$ and $\Gamma_2$ be two mixed graphs. If
\begin{itemize}[itemsep=4pt,parsep=0pt,topsep=4pt, partopsep=0pt]
\item $|V\Gamma_1| \neq |V\Gamma_2|$, or
\item $\Gamma_1$ and $\Gamma_2$ have different numbers of vertices appearing as origins of directed edges,
\end{itemize}
then $G_{\Gamma_1}$ and $G_{\Gamma_2}$ are not isomorphic.
\end{cor}

\begin{remark}
Two RAAG presentations give isomorphic groups if and only if the defining graphs are isomorphic (see \cite{droms1987isomorphisms}).
However, this is not the case for tRAAGs as non-isomorphic mixed graphs can give isomorphic groups (see Example \ref{isomorphic_traags_but_not_graphs}). The isomorphism problem of tRAAGs via their defining graph is still open. 
\end{remark}

\subsection{Cayley graphs and growth}

This subsection is motivated by the work of the author in \cite{antolin2022geodesic}; the formulas obtained there extend to tRAAGs as well. Similar to Lemma 2 in \cite{droms1993cayley} we have the following: 

\begin{theorem}\label{thm_cayley_graphs}
The Cayley graphs of a tRAAG and its underlying RAAG are isomorphic as undirected graphs.
\end{theorem}

As a first result, we show that reduced sequences of syllables represent geodesics. Recall that $S = V \sqcup V^{-1}$ is our monoid generating set.

\begin{lemma}
Let $(g_1, \ldots, g_r)$ be a reduced sequence of syllables, with $g_i = v_i^{s_i}$, where for all $1 \leq i \leq r$ we have $v_i \in V$, and $s_i\in \Z\setminus \{0\}$. Put $g = g_1 \cdots g_r$. Then the word~$g_1 \ldots g_r$ is a geodesic and: 
$$|g| = \sum_{i=1}^r |s_i|.$$
\end{lemma}

\begin{proof}
Let $w$ be a geodesic word (in a normal form) in $G_{\Gamma}$ which represents $g$. As $w$ is a word of minimal length representing $g$, we have $|g| = l(w) \leq l(g)$.
Since $(g_1, \ldots, g_r)$ is a reduced sequence of shortest syllable length representing $g$, we conclude that the syllable length of $w$ is equal to $r$. Write $w$ as $w = h_1 \ldots h_r$ where any $h_i$ is a syllable. Both sequences $(g_1, \ldots, g_r)$, and $(h_1, \ldots, h_r)$ are reduced and they represent the same element $g$. The normal form coming from CRS of tRAAGs implies that one can go from one to the other using only a series of shufflings of syllables (shufflings preserve lengths), i.e. the word $g_1 \ldots g_r$ is a geodesic and one has $l(g) = l(w) = |g|$, and the given formula holds.
\end{proof}
\begin{proof}[Proof of Theorem \ref{thm_cayley_graphs}] Let $G_{\Gamma}$ be our tRAAG based on the mixed graph $\Gamma$, and let~$G_{\overline{\Gamma}}$ be the underlying RAAG based on ${\overline{\Gamma}}$. 
Pick a total order in~$S = V \cup V^{-1}$.
Let $g = (g_1, \dots, g_n)$ be a reduced sequence of syllables in~$G_{\ga}$. 
Write its normal form with the smallest lexicographic order, both in $G_{\ga}$ and in $G_{\overline{\Gamma}}$; therefore we obtain two reduced sequences $w$ and $\overline{w}$ respectively. This representation of elements is unique. The map
$$i\colon G_{\Gamma} \longrightarrow G_{\overline{\Gamma}}  , \text{ given by } w \mapsto \overline{w},$$
gives a bijection of the underlying sets of elements; it maps $1$ to $1$ and $v^{\pm 1}$ to $v^{\pm 1}$ for any vertex~$v\in V\Gamma$, and it also preserves adjacencies (elements of distance $1$ in the Cayley graph are mapped to elements of distance $1$), as required.
\end{proof}
\begin{remark}
The tRAAG based on $\Gamma$ is quasi-isometric to its underlying RAAG.
\end{remark}
One of the aforementioned applications of the normal form theorem is to compare the spherical and the geodesic growth of tRAAGs with the respective growth of RAAGs. This follows from the discussion on Cayley graphs. 



\begin{theorem}
The spherical and the geodesic growth of a tRAAG over $\Gamma$ is equal to the corresponding growth of the RAAG based on the underlying graph $\overline{\Gamma}$.
\end{theorem}

\subsection{Torsion}\label{torsion_traags}

Twisted RAAGs have lots of similarities with RAAGs, especially on their geometric and combinatorial nature. However, when we work on their algebraic properties, we notice many remarkable differences: one was pointed out by the so-called {\it indicability} in Proposition \ref{prop: indicable tRAAGs}, and another one is the presence of torsion, as illustrated by the following example.

\begin{example}\label{ex: triangle with torsion}
The triangle $\Delta = (V, E)$, with~$V= \{a,b,c\}$, and~$E = \{[a,b\rangle, [b,c\rangle,[c,a\rangle\}$,
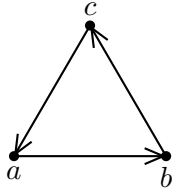
\begin{figure}[H]
\centering
\begin{tikzpicture}[>={Straight Barb[length=7pt,width=6pt]},thick]

\draw[fill=black] (0,0) circle (1.5pt) node[below] {$a$};
\draw[fill=black] (2,0) circle (1.5pt) node[below] {$b$};
\draw[fill=black] (1,{sqrt(3)}) circle (1.5pt) node[above] {$c$};

\draw[thick, ->] (0,0) -- (2,0);
\draw[thick, ->] (2,0) -- (1,{sqrt(3)});
\draw[thick, ->] (1,{sqrt(3)}) -- (0,0);

\end{tikzpicture}
\end{figure}
\noindent defines a tRAAG with presentation
$G_{\Delta} = \langle a, b, c | aba = b, bcb = c, cac = a\rangle.$ Any~$g \in G_{\Delta}$ can be represented as $g = a^{m}b^{n}c^{p}$ for some $m,n,p$ in $\mathbb{Z}$.
For $g = abc$ one has $g^2 = 1$, which means that there is torsion in this group.


\end{example}

\begin{mydef}
The {\it support} of $g \in \G$ is the set of vertices in $\Gamma$ used to write the normal form of $g$, and we denote it by $\supp(g)$. 
\end{mydef}

\begin{remark}\label{torsion_clique}
If $g \in \G$ has torsion then the vertices in $\supp(g)$ form a clique. Indeed, if two vertices in $\supp(g)$ are not connected, we cannot bring their corresponding syllables together using our rewriting system. 
\end{remark}

\begin{theorem}
$\G$ has torsion if and only if there is a clique $C$ in $\Gamma$ whose vertices form a closed cycle.
\end{theorem}
\begin{proof}
If there is such a clique $C$ then we have torsion. Indeed, for a clique of size $3$, see Example \ref{ex: triangle with torsion}. Otherwise, let $a_1, \ldots, a_n$ with~$n\geq 4$ be the vertices of a closed cycle inside a clique. Assume that $n$ is the smallest number with this property, which implies that $a_i, a_j$ commute whenever $|i-j| \geq 2$ (except of course $a_1, a_n$), as otherwise we would have a shorter cycle on a clique. One can pick $g\in G$ to be the product of vertices in this cycle in the order given by the cycle, and for~$g = a_1\cdots a_n$ we have $g^2 = 1$.

For the converse, assume that any clique in $\Gamma$ does not form a closed cycle, and there is~$g\in G$ which has torsion. By Remark \ref{torsion_clique}, we know that $\supp(g)$ lies in a clique~$C$ with $|C| \geq 3$. Assume that $g$ has the minimal number of elements in the support among the elements which have torsion. By assumption the clique $C$ does not form a closed cycle, nor does any of its sub-cliques. So, there is a vertex $v$ in $C$ which is not the origin of any directed edges. 
By Proposition \ref{prop: indicable tRAAGs} one has a map $\varphi:C\longrightarrow \mathbb{Z}$ with $\varphi(v) = 1$, and $\varphi(v') = 0$ for~$v' \in C \setminus \{v\}$, so~$\varphi(g)\neq 0$, and hence $g$ cannot have torsion.
\end{proof}

\vspace{1cm}
\noindent{\textbf{{Final remarks:}}}
\begin{itemize}
\item The rewriting system presented here would still work for a slightly more general class of groups, where one can also have relations of the form $xy = y^{-1}x^{-1}$. 
\end{itemize}

\noindent{\textbf{{Acknowledgments.}}}
The author would like to thank his doctoral supervisors Thomas Weigel and Yago Antolín, and his line manager Robert Gray for helpful suggestions and insights during the preparation of this project. 

The author gratefully acknowledges past support from the Department of Mathematics of the University of Milano-Bicocca, the Erasmus Traineeship grant 2020-1-IT02-KA103-078077, and current support from the EPSRC Fellowship grant EP/V032003/1 ‘Algorithmic, topological and geometric aspects of infinite groups, monoids and inverse semigroups’.

\addcontentsline{toc}{section}{Bibliography} 

\printbibliography

@article{antolin2022geodesic,
	title={Geodesic Growth of some 3-dimensional RACGs},
	author={Antol{\'\i}n, Yago and Foniqi, Islam},
	journal={Rocky Mountain Journal of Mathematics},
	volume={52},
	number={5},
	pages={1523--1537},
	year={2022},
	publisher={Rocky Mountain Mathematics Consortium Tempe, AZ, USA}
}

@article{cassella2024hypercubical,
	author = {Cassella, Alberto},
	publisher = {Università degli Studi di Milano-Bicocca},
	title = {Hypercubical groups},
	year = {2024}
}

@article{hermiller1995algorithms,
	author = {Hermiller, Susan and Meier, John},
	publisher = {Citeseer},
	title = {Algorithms and geometry for graph products of groups},
	year = {1995}
}

@book{holt2005handbook,
	author = {Holt, Derek F and Eick, Bettina and O'Brien, Eamonn A},
	publisher = {Chapman and Hall/CRC},
	title = {Handbook of computational group theory},
	year = {2005}
}

@article{chouraqui2009rewriting,
  title={Rewriting systems and embedding of monoids in groups},
  author={Chouraqui, Fabienne},
  year={2009},
  publisher={Walter de Gruyter GmbH \& Co. KG}
}

@article{blumer2023oriented,
	author = {Blumer, S. and Quadrelli, C. and Weigel, T.},
	journal = {arXiv preprint arXiv:2304.08123},
	title = {Oriented right-angled Artin pro-$\ell$ groups and maximal pro-$\ell$ Galois groups},
	year = {2023}
}

@article{clancy2010homology,
	author = {Clancy, Maura and Ellis, Graham},
	journal = {Journal of K-Theory},
	number = {1},
	pages = {171–196},
	publisher = {Cambridge University Press},
	title = {Homology of some Artin and twisted Artin Groups},
	volume = {6},
	year = {2010}
}

@article{droms1993cayley,
	author = {Droms, Carl and Servatius, Herman},
	journal = {Proceedings of the American Mathematical Society},
	number = {3},
	pages = {693–698},
	title = {The Cayley graphs of Coxeter and Artin groups},
	volume = {118},
	year = {1993}
}

@phdthesis{foniqi2022,
	author = {Foniqi, Islam},
	pages = {},
	school = {University of Milano - Bicocca},
	title = {Results on Artin and twisted Artin groups},
	url = {https://boa.unimib.it/handle/10281/374264},
	year = {2022}
}

@phdthesis{phdthesis,
	author = {Blumer, Simone and Weigel, Thomas},
	doi = {10.13140/RG.2.2.16287.07843},
	month = {09},
	pages = {},
	school = {University of Milano - Bicocca},
	title = {TEORIA GEOMETRICA DEI GRUPPI: Spazi CAT(0), Teorema di Gromov e Oriented Right-Angled Artin Groups},
	year = {2020}
}

@article{charney2007introduction,
	author = {Charney, Ruth},
	journal = {Geometriae Dedicata},
	number = {1},
	pages = {141–158},
	publisher = {Springer},
	title = {An introduction to right-angled Artin groups},
	volume = {125},
	year = {2007}
}

@article{droms1987isomorphisms,
	author = {Droms, Carl},
	journal = {Proceedings of the American Mathematical Society},
	number = {3},
	pages = {407–408},
	title = {Isomorphisms of graph groups},
	volume = {100},
	year = {1987}
}

\noindent\textit{\\Islam Foniqi,\\
University of East Anglia\\ 
Norwich, UK\\}
{email: i.foniqi@uea.ac.uk}

\pgfdeclaredecoration{free hand}{start}
{
  \state{start}[width = +0pt,
                next state=step,
                persistent precomputation = \pgfdecoratepathhascornerstrue]{}
  \state{step}[auto end on length    = 3pt,
               auto corner on length = 3pt,               
               width=+2pt]
  {
    \pgfpathlineto{
      \pgfpointadd
      {\pgfpoint{2pt}{0pt}}
      {\pgfpoint{rand*0.15pt}{rand*0.15pt}}
    }
  }
  \state{final}
  {}
}
 \tikzset{free hand/.style={
    decorate,
    decoration={free hand}
    }
 } 
\def\freedraw#1;{\draw[free hand] #1;}

\end{document}